\documentclass[preprint]{elsarticle}

\usepackage[english]{babel}
\usepackage[T1]{fontenc}
\usepackage[utf8]{inputenc}
\usepackage{lmodern}

\usepackage[draft]{fixme}
\fxsetup{theme=color}
\usepackage{mathtools}
\usepackage{commath}
\usepackage{natbib}
\usepackage{esvect}
\DeclareMathOperator{\sign}{sign}

\usepackage{graphicx}
\usepackage{amsmath,amssymb,amsfonts,amsthm}

\newtheorem{theorem}{Theorem}
\newtheorem{lemma}{Lemma}
\newtheorem{proposition}{Proposition}
\newtheorem{corollary}{Corollary}
\theoremstyle{definition}
\newtheorem{definition}{Definition}
\theoremstyle{remark}

\newtheorem*{remark*}{Remark}

\usepackage{hyperref}
\hypersetup{
  pdfauthor={Martin Lukarevski, Hans-Peter Schröcker},pdftitle={Exparabolas of a Triangle},pdfborder={0 0 0},colorlinks=true,}

\newcommand{\AR}[3]{\operatorname{ar}(#1,#2,#3)}
\newcommand{\axisDir}{\vv{a}}

\begin{document}

\begin{frontmatter}
\journal{}

\title{Exparabolas of a Triangle}
\author{Martin Lukarevski}
\ead{martin.lukarevski@ugd.edu.mk}
\address{Goce Delčev University of Štip, Department of Mathematics and Statistics, Štip, North Macedonia}

\author{Hans-Peter Schröcker}
\ead{hans-peter.schroecker@uibk.ac.at}
\address{University of Innsbruck, Department of Basic Sciences in Engineering Sciences, Innsbruck, Austria}

\begin{abstract}
  Among a triangle's exparabolas (parabolas escribed to the triangle), three are distinguished by having locally maximal parameter. They are determined by a simple cubic equation and characterized by having axes that contain the triangle's centroid. More generally, there are three (not necessarily real) exparabolas with axes through a given point $X$. Their focal points determine another triangle which we call the $X$-focal triangle. It shares the circumcircle with the original triangle and its orthocenter is $X$. The sequence of iterated focal triangles with respect to the centroids splits into an even and an odd sub-sequence that both converge to equilateral triangles.
\end{abstract}


\end{frontmatter}

\section{Introduction}
\label{sec:introduction}

In the article \cite{lukarevski25} we studied maximal parabolas in Euclidean geometry and minimal horocycles in hyperbolic geometry. As a technical tool we introduced parabolas tangent to the sides of a triangle (``exparabolas'' of the triangle) and exparabolas that are of minimal parameter (``max-exparabolas''). They turned out to come in triples, each opposite to one triangle vertex, meaning that they are tangent to the opposite triangle side in an interior point.

In this article we study exparabolas within triangle geometry and focus in particular on exparabolas of maximal parameter and triples of exparabolas with concurrent axes. In re-proving the existence and uniqueness result of \cite{lukarevski25} in Theorem~\ref{th:max-exparabolas} we derive simple cubic equations with coefficients depending on the triangle sides only, whose roots describe the three max-exparabolas. We characterize max-exparabolas among all exparabolas as having axes incident with the triangle's centroid in Corollary~\ref{cor:centroid}. This follows immediately from a more general statement on exparabolas with axes through a given point $X$ (Theorem~\ref{th:X-exparabola}).

The focal triangle with respect to a point $X$ is formed by the focal points of three exparabolas with axes concurrent at $X$. The three axes are its altitudes (Theorem~\ref{th:focal-orthocenter}) so that $X$ is the orthocenter of the $X$-focal triangle. In case of $X$ being the centroid, we recover the max-exparabolas. We use this property to investigate the limiting behavior of the sequence of iterated focal triangles with respect to the centroid. Its even and odd subsequence, respectively, converge to equilateral triangles. The six vertices of these two triangles form a regular hexagon (Corollary~\ref{cor:focal-sequence}).

Throughout this text we will use some formulas to compute metric data (vertices, axis directions, focal points, parameters) of parabolas given in parametric form as Bézier curves of degree two. These formulas are sometimes a bit long and their derivation is technical so that we collect them in an appendix with the aim of not disturbing too much the flow of reading.

\section{Parabolas Escribed to a Triangle}
\label{sec:escribed-parabolas}

Consider a triangle $ABC$ with vertices $A$, $B$, $C$ and opposite edges $BC$, $CA$, $AB$. The distance between triangle vertices are denoted as usual by $a = \vert BC \vert$, $b = \vert CA \vert$, and $c = \vert AB \vert$.

\begin{definition}[\cite{lukarevski25}]
  We call any parabola $p$ that is tangent to the three sides of the triangle $ABC$ an \emph{exparabola} of $ABC$. If $p$ is tangent to $AB$ in a point between $A$ and $B$ we say that $p$ is \emph{opposite to $C$.}
\end{definition}

The term ``exparabola'' is motivated because it describes a concept that is reminiscent to the excircles of a triangle. Note however, that there are infinitely many exparabolas to a triangle while there are only three excircles.

\begin{remark*}
The set of exparabolas concists precisely of the regular conics in the dual pencil of conics that are tangent to the three sides of the triangle $ABC$ and to the line at infinity. The singular conics in this pencil divide this dual pencil of exparabolas into three subsets, each consisting of all exparabolas opposite to one vertex.
\end{remark*}

An exparabola $p$ is uniquely determined by its point of tangency (different from the triangle vertices) on any of the triangle's sides. This means that there is a dependence between the three points of tangency. It is described in the following lemma in terms of homogeneous barycentric coordinates \cite[Section~9.1]{UniverseConics}.

\begin{lemma}
  \label{lem:barycentric-coordinates}
  For any exparabola there exists $t \in \mathbb{R} \setminus \{0,1\}$ such that its points of tangency with the triangle sides have homogeneous barycentric coordinates
  \begin{equation}
    \label{eq:1}
    A_0 \coloneqq (0,1,t-1),\quad
    B_2 \coloneqq (1,0,-t),\quad
    C_1 \coloneqq (1-t,t,0),
  \end{equation}
  cf. Figure~\ref{fig:escribed-parabola}. (The reason for the seemingly strange labeling of points will become clear later, when we consider three point triples of above type.)
\end{lemma}

The statement of Lemma~\ref{lem:barycentric-coordinates} follows from well-known properties of parabolas if we observe that
\[\AR{B_2}{C}{A} = \AR{A}{B}{C_1} = \AR{C}{A_0}{B} = \frac{t}{1-t}\]
where
\[
  \AR{X}{Y}{Z} \coloneqq \frac{\vv{XZ}}{\vv{ZY}}
\]
is the \emph{affine ratio} of three collinear points $X$, $Y$, and $Z$ and the arrows denote signed distances. In fact, we may view the points $A$, $B$, and $C_1$ as intermediate points constructed by de Casteljau's algorithm when applied to the Bézier curve with control points $B_2$, $C$, and~$A_0$ \cite[Section~3.1]{farin97}.

\begin{figure}
  \centering
  \includegraphics[]{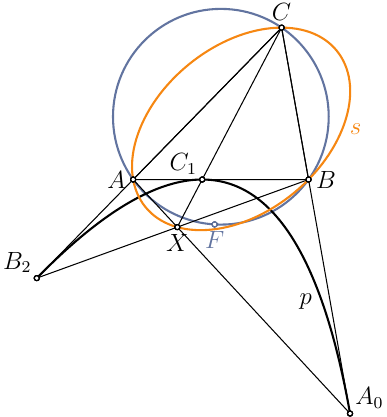}
  \caption{Parabola escribed to triangle $ABC$, Steiner circumellipse $s$, and focal point $F$ on the circumcircle.}
  \label{fig:escribed-parabola}
\end{figure}

For later reference, we state some basic properties of exparabolas:

\begin{lemma}
  \label{lem:exparabola-basic}
  If $p$ is an exparabola of the triangle $ABC$ with points of tangency given by
  \eqref{eq:1}, then:
  \begin{itemize}
  \item The lines $AA_0$, $BB_2$, and $CC_1$ intersect in a point $X$ on the Steiner circumellipse~$s$.
  \item The focal point $F$ of $p$ lies on the circumcircle of~$ABC$.
  \end{itemize}
\end{lemma}

\begin{proof}
  The first part of the first statement follows from Ceva's Theorem because the product of the affine ratios
  \[
    \AR{A}{B}{C_1} = \frac{t}{1-t},\quad
    \AR{B}{C}{A_0} = t - 1,\quad
    \AR{C}{A}{B_2} = -\frac{1}{t}
  \]
  equals $1$. The barycentric coordinate vector of the intersection point are $X = (t-1, -t, t(1-t))$. It satisfies the equation $xy + yz + zx = 0$ of the Steiner circumellipse \cite{kimberling20}. The second claim is well known, cf. for example \cite{Dragovic25} for a computational proof or \cite[Theorem~9.1.2]{UniverseConics}.
\end{proof}

\begin{remark*}
  The point $X$ in Lemma~\ref{lem:exparabola-basic} is the Brianchon point in the three-tangents degeneration case of Brianchon's Theorem \cite[Theorem 6.2.3 and Figure~6.7]{UniverseConics}.
\end{remark*}

\section{Exparabolas of Maximal Parameter}
\label{sec:maximal-parameter}

In the one-parametric family of exparabolas of triangle $ABC$, given by their points of tangency \eqref{eq:1} for varying $t$, we want to single out particular ones.

\begin{definition}
  \label{def:exparabola}
  An exparabola of the triangle $ABC$ which locally has a maximal parameter is
  called a \emph{max-exparabola} of~$ABC$.
\end{definition}

Recall that parameter of a conic equals the product of numerical eccentricity times the distance between focal point and directrix in the Apollonian definition of conics \cite[Definition~2.1.1]{UniverseConics}. For our purposes it is more convenient to regard it as the curvature radius in the parabola vertex \cite[Theorem~2.1.5]{UniverseConics}. In \cite{lukarevski25} it served as a measure for the size of a parabola. Hence, max-exparabolas may also be addressed as locally maximal escribed parabolas of $ABC$.

\begin{theorem}
  \label{th:max-exparabolas}
  A triangle $ABC$ has three max-exparabolas, one opposite to each of the three
  triangle vertices (Figure~\ref{fig:max-exparabolas}).
\end{theorem}

\begin{figure}
  \centering
  \includegraphics[]{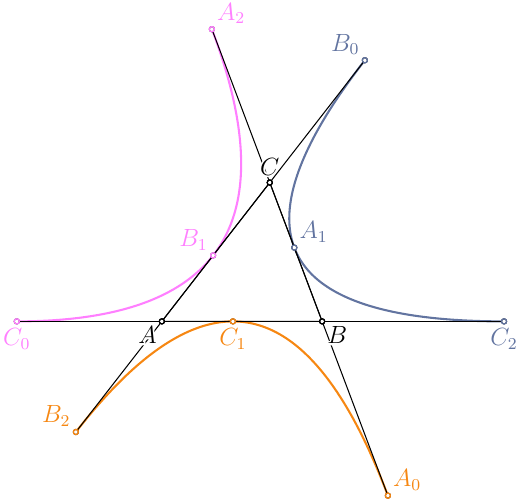}
  \caption{Exparabolas of triangle $ABC$.}
  \label{fig:max-exparabolas}
\end{figure}

\begin{proof}
  We parametrize an exparabola $p$ as quadratic Bézier curve
  \begin{equation}
    \label{eq:2}
    p\colon (1-u)^2 B_2 + 2(1-u)u C + u^2 A_0,\quad
    u \in \mathbb{R}
  \end{equation}
  (cf. \cite[Section~4.1]{farin97}) where the points $B_2$, $C$, and $A_0$ are as in \eqref{eq:1}. In addition to the triangle vertices, they depend on a parameter $t$ that is to be determined. The squared parameter $\varrho^2$ can be computed by Equation~\eqref{eq:parameter2} in the appendix. With $A = (0,0)$, $B = (c, 0)$, and $C = (c_1,c_2)$ we find
  \[
    \varrho^2 = \frac{4c^4c_2^4t^2(t-1)^2}
    {(c^2t^2+2c(c_1-c)t+c^2-2cc_1+c_1^2+c_2^2)^3}.
  \]
  In order to determine the stationary values of $\varrho^2$, we solve \[ \od{\varrho^2}{t} = 0 \] for $t$. The numerator of the right-hand side factors as $-8c^4c_2^4t(t-1)e$. Since $c$, $c_2$, $t$ and $t-1$ where
  \[
    e = c^2t^3-c(c+c_1)t^2+(-c^2+3cc_1-2c_1^2-2c_2^2)t+c^2-2cc_1+c_1^2+c_2^2.
  \]
  Since $c$, $c_2$, $t$, and $t-1$ are all different from zero, stationary values of the parameter are the roots of the cubic equation $e = 0$. Thus, there are \emph{at most} three exparabolas with stationary parameter.

  The cubic $e$ will be much simpler, closer to standard notation of triangle geometry, and, most importantly, independent of the particular coordinates we chose for $A$, $B$, and $C$, if we express it in terms of the triangle side lengths $a$, $b$, and $c$. For that purpose, we denote the angle of the triangle at $A$ by $\alpha$ and, using the law of cosines, substitute
  \begin{equation}
    \label{eq:3}
    c_1 = b\cos\alpha,\quad
    c_2 = b\sqrt{1-\cos^2\alpha}
  \end{equation}
  where \[\cos\alpha = \frac{b^2+c^2-a^2}{2bc}\] and simplify to obtain
  \begin{equation}
    \label{eq:4}
    e = c^2t^3+\tfrac{1}{2}(a^2-b^2-3c^2)t^2-\tfrac{1}{2}(3a^2+b^2-c^2)t+a^2.
  \end{equation}
  Observe that
  \begin{equation*}
    \lim_{t\to-\infty}e(t) = -\infty,\quad
    e(0) = a^2 > 0,\quad
    e(1) = -b^2 < 0,\quad
    \lim_{t\to\infty}e(t) = \infty.
  \end{equation*}
  and apply the intermediate value theorem to see that there are \emph{three real roots} $t_0$, $t_1$, $t_2$ that interlace with $0$ and $1$ as \(t_0 < 0 < t_1 < 1 < t_2\). From this we infer that there are \emph{precisely} three exparabolas of stationary parameter, one opposite to each vertex. Finally, the infimum $0$ of the parameter in the family of exparabolas is approached for $t \to 0$, $t \to 1$, or $t \to \infty$. Hence, the parameter is not just stationary but locally maximal.
\end{proof}

Theorem~\ref{th:max-exparabolas} has already been proved in \cite{lukarevski25} but with slightly different computations. The simple cubic \eqref{eq:4} is new.

In order to introduce more symmetry and structure into our notation, we denote, for $i \in \{0,1,2\}$ the point of tangency corresponding to the root $t_i$ via \eqref{eq:1} by $C_i$. Each parameter value $t_i$ (equivalently, each point of tangency $C_i$) corresponds to one max-exparabola which we denote by $p_0$, $p_1$, and $p_2$, according to the index of the root $t_i$ or by~$p_B$, $p_C$, and $p_A$, according to the triangle vertex opposite to it.

The points of tangency of the max-exparabolas on the remaining triangle sides can be found from the two cubics that are obtained by cyclically permuting the triple $(a,b,c)$ in \eqref{eq:4}. Setting $e_c \coloneqq 2e$, the three cubics are
\begin{equation}
  \label{eq:5}
  \begin{aligned}
    e_a &= 2a^2t^3 + (b^2-c^2-3a^2)t^2 - (3b^2+c^2-a^2)t + 2b^2,\\
    e_b &= 2b^2t^3 + (c^2-a^2-3b^2)t^2 - (3c^2+a^2-b^2)t + 2c^2,\\
    e_c &= 2c^2t^3 + (a^2-b^2-3c^2)t^2 - (3a^2+b^2-c^2)t + 2a^2.
  \end{aligned}
\end{equation}
They give rise to points of tangency $A_0$, $A_1$, $A_2$ on $BC$, and $B_0$, $B_1$, $B_2$ on $CA$, respectively.

Of course, the roots of $e_a$, $e_b$, and $e_c$ are closely related as the points of tangency $A_0$, $A_1$, $A_2$ on $BC$ and $B_0$, $B_1$, $B_2$ on $CA$ are already determined by $C_0$, $C_1$, $C_2$, that is, by the roots of $e_c$. The homogeneous barycentric coordinates of all these points in terms of the roots $t_0$, $t_1$, $t_2$ of $e_c$ are as follows:
\begin{alignat}{3}
  A_0 &= (0, 1, t_1-1),\quad&
  A_1 &= (0, 1, t_2-1),\quad&
  A_2 &= (0, 1, t_0-1), \notag\\
  \label{eq:6}
  B_0 &= (1, 0, -t_2),\quad&
  B_1 &= (1, 0, -t_0),\quad&
  B_2 &= (1, 0, -t_1),\\
  C_0 &= (1-t_0,t_0,0),\quad&
  C_1 &= (1-t_1,t_1,0),\quad&
  C_2 &= (1-t_2,t_2,0). \notag
\end{alignat}

Before we continue to explore Euclidean properties of exparabolas in general and max-exparabolas in particular in Section~\ref{sec:exparabola-axes}, let is state the obvious affine consequences of Lemma~\ref{lem:exparabola-basic}.

\begin{proposition}
  For a triangle $ABC$ with exparabolas $p_A$, $p_B$, $p_C$ opposite to $A$, $B$, and $C$, respectively, denote their points of tangency with the triangle sides as in Figure~\ref{fig:exparabola-triangle}. By Lemma~\ref{lem:exparabola-basic},
  they give rise to three points
  \begin{equation*}
    X_1 \coloneqq BB_0 \cap AA_1 \cap CC_2,\quad
    X_2 \coloneqq CC_0 \cap BB_1 \cap AA_2,\quad
    X_3 \coloneqq AA_0 \cap CC_1 \cap BB_2
  \end{equation*}
  which form a triangle inscribed into the Steiner circumellipse~$s$.
\end{proposition}

\begin{figure}
  \centering
  \includegraphics[]{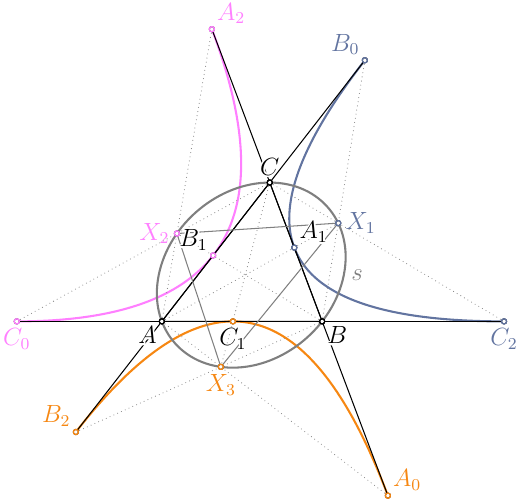}
  \caption{A triangle inscribed into the Steiner circumellipse.}
  \label{fig:exparabola-triangle}
\end{figure}

\section{The Axes of Exparabolas}
\label{sec:exparabola-axes}

In this section we introduce a new concept related to the exparabolas of a
triangle $ABC$:

\begin{theorem}
  \label{th:X-exparabola}
  Given a triangle $ABC$ and a point $X$ with normalized barycentric
  coordinates $(x_0,x_1,x_2)$ that satisfy
  \begin{equation}
    \label{eq:7}
    (x_0+x_1)(x_1+x_2) > 0,\quad
    (x_0+x_2)(x_1+x_2) > 0,\quad
    (x_0+x_1)(x_0+x_2) > 0,
  \end{equation}
  there exist precisely three exparabolas, each opposite to one of the triangle
  vertices, whose axes are incident with~$X$.
\end{theorem}

\begin{proof}
  We parametrize an exparabola $p$ as in \eqref{eq:2}, where the points $B_2$
  and $A_0$ depend on a yet undetermined barycentric weight~$t$, that is,
  \[
    A_0 = \frac{1}{t}(B + (t-1)C),\quad
    B_2 = \frac{1}{1-t}(A - tC).
  \]
  With $A = (0, 0)$, $B = (c, 0)$, and $C = (c_1,c_2)$, the vertex of $p$,
  computed via \eqref{eq:vertex_value}, is
  \begin{multline*} 
  V = \frac{ct}{(c^2(1-t)^2-2cc_1(1-t)+c_1^2+c_2^2)^2} \\
  \begin{pmatrix}
  c_1^4+c_2^4
  -3cc_1^3(1-t)
  -c_1^2(3c^2(1-t)^2-2c_2^2)
  -cc_1(1-t)(2c_2^2-(1-t)^2c^2)
  \\
  -cc_2(1-t)(c_1-c(1-t))^2
  \end{pmatrix}.
  \end{multline*} 
  Equation~\ref{eq:axis_direction} yields the axis direction
  \[
    \axisDir = \frac{1}{t(1-t)}
    \begin{pmatrix} (1-t)c-c_1 \\ -c_2 \end{pmatrix}.
  \]
  The point $X$ lies on the parabola axis if and only if $\det(V-X,\axisDir) = 0$. This condition depends on the coordinates of $A$, $B$, $C$ (that is, on $c$, $c_1$, and $c_2$) and the barycentric coordinates $(x_0,x_1,x_2)$ of $X$ with respect to $ABC$. After multiplying away the denominator, dividing off irrelevant factors, and substituting \eqref{eq:3} we obtain the condition $f = 0$ where
  \begin{multline}
    \label{eq:8}
    f = c^2(x_0+x_1)t^3 + ((a^2-b^2)x_2 - (2x_0+x_1)c^2)t^2\\
      + ((c^2-b^2)x_0 - (2x_2+x_1)a^2)t + a^2(x_1+x_2).
  \end{multline}
  It is cubic in $t$ and there exist at most three exparabolas with the required
  properties. Existence of exactly three real solutions follows from
  \begin{equation*}
      \lim_{t\to\pm\infty}f(t) = \pm \sign(x_0+x_1)\infty,\quad
      f(0) = a^2(x_1+x_2),\quad
      f(1) = -b^2(x_0+x_2)
  \end{equation*}
  together with the assumptions \eqref{eq:7} of the theorem.
\end{proof}

\begin{figure}
  \centering
  \includegraphics[]{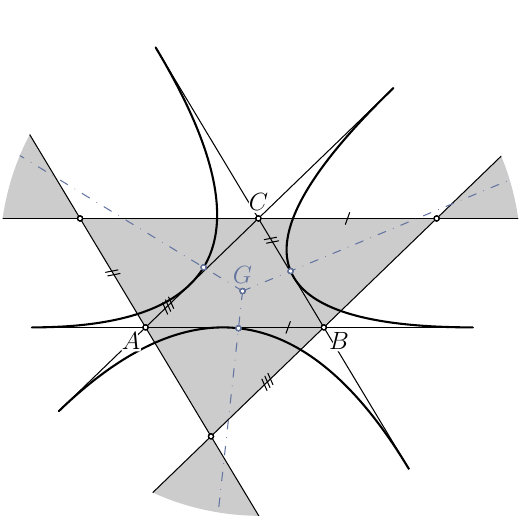}
  \caption{Region for admissible points $X$ in Theorem~\ref{th:X-exparabola}.
    The axes of the max-exparabolas intersect in the triangle's centroid $G$.}
  \label{fig:anticomplementary}
\end{figure}

The admissible region for points $X$ in Theorem~\ref{th:X-exparabola} is bounded
by the lines of the anticomplementary triangle, c.f.
Figure~\ref{fig:anticomplementary}. The case of $X$ being the triangle's
centroid yields a new characterization of the max-exparabolas:

\begin{corollary}
  \label{cor:centroid}
  For an exparabola $p$ of a triangle, the following two statements are
  equivalent:
  \begin{itemize}
  \item $p$ is a max-exparabola.
  \item The axis of $p$ contains the triangle's centroid $G$.
  \end{itemize}
  In particular, the axes of the three max-exparabolas intersect in the
  triangle's centroid $G$ (Figure~\ref{fig:anticomplementary}).
\end{corollary}

\begin{proof}
  For $x_0 = x_1 = x_2 = 1$, the cubics \eqref{eq:4} and \eqref{eq:8} have the same roots.
\end{proof}

Three exparabolas with concurrent axes have give rise to a focal triangle and a theorem related to it.

\begin{definition}
Given three different parabolas, we define their \emph{focal triangle} as the triangle formed by their respective foci. If the three parabolas are exparabolas of a triangle $ABC$ we speak of \emph{a focal triangle of $ABC$.} Finally, if $X$ is an arbitrary point and the three exparabolas have axes incident with $X$, we call the focal triangle the \emph{focal triangle of $ABC$ with respect to $X$} or the \emph{$X$-focal triangle} of $ABC$ for short.
\end{definition}

The focal triangle is always well-defined unless two of the three exparabolas coincide:
\begin{itemize}
\item As an immediate consequence of Lemma~\ref{lem:exparabola-basic}, the circumcircle of $ABC$ is also a circumcircle of the focal triangle. In particular, the three focal points are never collinear.
\item Even if the three exparabolas with axes through $X$ are not real (the barycentric coordinates of $X$ with respect to $ABC$ do not satisfy \eqref{eq:7}), their foci are.
\end{itemize}

\begin{theorem}
  \label{th:focal-orthocenter}
  For any triangle $ABC$ and any point $X$, the orthocenter of the $X$-focal triangle is the point $X$ itself.
\end{theorem}

\begin{proof}
  We consider three exparabolas $p_A$, $p_B$, $p_C$ given by tangent points with respective homogeneous barycentric coordinates
  \begin{alignat*}{3}
    A_0 &= (0, 1, v-1),\quad&
    A_1 &= (0, 1, w-1),\quad&
    A_2 &= (0, 1, u-1),\\
    B_0 &= (1, 0, -w),\quad&
    B_1 &= (1, 0, -u),\quad&
    B_2 &= (1, 0, -v),\\
    C_0 &= (1-u,u,0),\quad&
    C_1 &= (1-v,v,0),\quad&
    C_2 &= (1-w,w,0)
  \end{alignat*}
  for pairwise different $u$, $v$, $w \in \mathbb{R} \setminus \{0,1\}$.
  Using once more the simple coordinates
  \[
    A = (0, 0),\quad
    B = (c, 0),\quad
    C = (c_1, c_2)
  \]
  and Equation~\eqref{eq:focus} from the appendix, we obtain the focal point
  \[
    F_A = \frac{cw}{c^2(1-w)^2-2cc_1(1-w)+c_1^2+c_1^2}
    \begin{pmatrix}
      c_1^2+c_2^2-cc_1(1-w)\\
      -cc_2(1-w)
    \end{pmatrix}
  \]
  and, via \eqref{eq:axis_direction}, the
  corresponding axis direction
  \[
    \axisDir_A = \begin{pmatrix} c(w-1)+c_1 \\ c_2 \end{pmatrix}.
  \]
  The formulas for the focal points $F_B$, $F_C$ and for their axis directions $\axisDir_B$, $\axisDir_C$ are obtained by replacing $w$ by $u$ and $t$ respectively. The conditions for perpendicularity of one exparabola axis to the connecting line of the other two focal points read
  \begin{equation*}
    \begin{aligned}
      h \coloneqq &\det\begin{pmatrix}\axisDir_A,F_B-F_C\end{pmatrix} = 0,\\
                  &\det\begin{pmatrix}\axisDir_B,F_C-F_A\end{pmatrix} = 0,\\
                  &\det\begin{pmatrix}\axisDir_C,F_A-F_B\end{pmatrix} = 0.
    \end{aligned}
  \end{equation*}
  It turns out, that all three conditions are the same. In other words, if $h = 0$, then not only the axis of $p_A$ is perpendicular to the line $F_BF_C$ but also the axes of $p_B$ and $p_C$ are perpendicular to the lines $F_CF_A$ and $F_AF_B$, respectively. We therefore focus on the condition $h = 0$. Substituting \eqref{eq:3}, it becomes
  \begin{multline}
    \label{eq:9}
    h = c^2((a^2+b^2)c^2-(a-b)^2(a+b)^2)uvw
    + a^2c^2(a^2-b^2-c^2)(uv+vw+wu) \\
    - a^2c^2(a^2+b^2-c^2)(u+v+w)
    + a^2((b^2+c^2)a^2-(b-c)^2(b+c)^2).
  \end{multline}
  Assume now that $u$, $v$ and $w$ are the roots of \eqref{eq:8}. Then
  \begin{equation*}
    \begin{aligned}
    uvw &= -\frac{a^2(x_1+x_2)}{c^2(x_0+x_1)},\\
    uv+vw+wu &= -\frac{a^2(x_1+2x_2)+(b^2-c^2)x_0}{c^2(x_0+x_1)},\\
    u+v+w &= \phantom{-}\frac{c^2(2x_0+x_1)+(b^2-a^2)x_2}{c^2(x_0+x_1)}.
    \end{aligned}
  \end{equation*}
  Plugging these into the right-hand side of \eqref{eq:9} gives $0$ so that $h =
  0$ is indeed fulfilled.
\end{proof}

\section{The Focal Triangle of the Max-exparabolas}
\label{sec:focal-triangle}

We now turn to the focal triangle of the max-exparabolas. Corollary~\ref{cor:centroid} together with Theorem~\ref{th:focal-orthocenter} yield:

\begin{figure}
  \centering
  \includegraphics[]{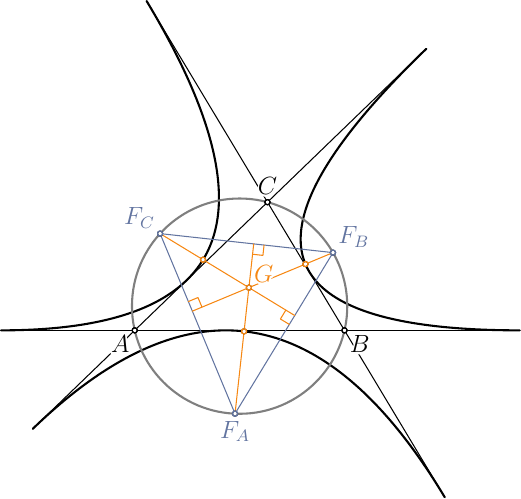}
  \caption{The altitudes of the focal triangle $F_AF_BF_C$ are the axes of the
    max-exparabolas.}
  \label{fig:focaltriangle}
\end{figure}

\begin{corollary}
  \label{cor:focal-orthocenter}
  For any triangle $ABC$ the following hold true:
  \begin{itemize}
  \item The orthocenter of the focal $F_AF_BF_C$ triangle with respect to the max-exparabolas is the centroid $G$ of $ABC$.
  \item The axes of the max-exparabolas of $ABC$ are the altitudes of $F_AF_BF_C$ and intersect in~$G$.
  \item The triangles $ABC$ and $F_AF_BF_C$ have the same Euler line.
  \end{itemize}
\end{corollary}

\begin{remark*}
The result for the orthocenter of the focal triangle is analogous to the result in triangle geometry for the excentral triangle. The orthocenter of the excentral triangle is the incenter $I$ of the triangle, see \cite{luka20} where some inequalities for the altitudes of the excentral triangle are given.
\end{remark*}

We now use Corollary~\ref{cor:focal-orthocenter} to investigate the limiting behavior of a sequence of iterated focal triangles with respect the centroid.

\begin{corollary}
  \label{cor:focal-sequence}
  Given a triangle $A_0B_0C_0$, denote for $i \ge 0$ the centroid of $A_iB_iC_i$ by $G_i$ and the $G_i$-focal triangle of $A_iB_iC_i$ by $A_{i+1}B_{i+1}C_{i+1}$. Then the limits
  of the sequences
  \[
    A_{2i}B_{2i}C_{2i}
    \quad\text{and}\quad
    A_{2i+1}B_{2i+1}C_{2i+1}
  \]
  are both equilateral triangles that together form the vertices of a regular
  hexagon (Figure~\ref{fig:focal-sequence}).
\end{corollary}

\begin{proof}
  Denote by $(O_i,G_i,H_i)$ the triple consisting of circumcenter, centroid and orthocenter of $A_iB_iC_i$. Then $O_i = O_0$ is constant and all triangles have the same circumcircle. Moreover, $G_{i+1}$ and $H_{i+1}$ are obtained from $G_i$ and $H_i$, respectively, by a scaling from $O_0$ with scale factor $\frac{1}{3}$ because
  \begin{itemize}
  \item the centroid divides the segment between circumcenter and orthocenter with ratio $1:2$ in every triangle \cite[Thereom~201]{Altshiller-Court07} and
  \item $H_{i+1} = G_i$ (Theorem~\ref{th:focal-orthocenter}).
  \end{itemize}
  Therefore, we obtain a sequence of triangles inscribed in the circumcircle of $A_0B_0C_0$ whose centroids converge to the circumcircle center. However, a triangle whose circumcenter and centroid coincide is equilateral. Now the claim follows from the observation that the vertices of an equilateral triangle and its focal triangle with respect to the max-exparabolas form the even and odd vertices of a regular hexagon.
\end{proof}

\begin{figure}
  \centering
  \includegraphics[page=1]{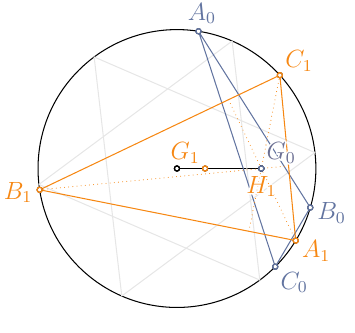}
  \includegraphics[page=2]{./img/focal-sequence}\\
  \includegraphics[page=3]{./img/focal-sequence}
  \includegraphics[page=4]{./img/focal-sequence}
  \caption{Iterated sequence of focal triangles. In each step, the original
    triangle is in blue and the focal triangle is in orange, the limiting
    equilateral triangles are in gray. Centroids are denoted by $G_i$, orthocenters
    by $H_i$.}
  \label{fig:focal-sequence}
\end{figure}

Figure~\ref{fig:focal-sequence} shows the first triangles in a sequence of focal triangles with respect to max-exparabolas. The convergence stated in Corollary~\ref{cor:focal-sequence} can clearly be observed.

\section{Conclusion}
\label{sec:conclusion}

In this article we introduced exparabolas and focal triangles associated to triples of exparabolas to triangle geometry. In particular we studied max-exparabolas and their focal triangles. Max-exparabolas can also be characterized as exparabolas whose axes contain the triangle's centroid $G$. Iteratively constructing focal triangles with respect to max-exparabolas yields a remarkable sequence of triangles that becomes 2-periodic in the limit.

It is to be expected that further interesting triangle properties are related to exparabolas in general and to special exparabolas. Besides the focal triangle one can also study the triangle formed by three exparabola axis. Furthermore, iterated sequences can not just be based on the triangle centroids but also on other points. We have first numerical evidence that the 2-periodic limit is rather typical but the limiting shapes differ.

\appendix
\section{Formulas for Bézier Parabolas}

In this appendix we collect some technical calculations and formulas for parametric parabolas
given as quadratic Bézier curves
\begin{equation}
  \label{eq:10}
  p\colon P(u) =
  (1-u)^2P_0 + 2(1-u)uP_1 + u^2P_2
\end{equation}
with control points $P_0 = (x_0,y_0)$, $P_1 = (x_1,y_1)$, $P_2 = (x_2,y_2)$ (cf.~\cite[Section~4.1]{farin97}. The formulas are frequently referred to at several places in this text.

The parabola $p$ has a unique point at infinity. It is obtained from \eqref{eq:10} in the limit for $u \to \pm\infty$. Therefore, the \emph{axis direction $\axisDir$} of $p$ is the leading coefficient of $P(u)$:
\begin{equation}
  \label{eq:axis_direction}
  \axisDir = P_0 - 2P_1 + P_2.
\end{equation}

In order to obtain the \emph{vertex $V$} of $p$, we compute
\[
  \dot{P}(u) = 2(P_0 - 2P_1 + P_2)u +
  2(P_1 - P_0)
\]
and solve $\langle \axisDir, \dot{P}(u) \rangle$ for $u$. This gives the \emph{parameter value $u_V$} of the vertex $V$:
\begin{equation}
  \label{eq:vertex_value}
  u_V = -\frac{\langle \axisDir, P_1-P_0 \rangle}{\Vert \axisDir \Vert^2}.
\end{equation}
The vertex itself is obtained as $V = P(u_V)$.

The \emph{parameter} $\varrho$ of $p$ equals the curvature radius in $V$, c.f. \cite[Theorem~2.1.5]{UniverseConics}. The squared curvature $\varkappa^2$ of the parabola is computed by the formula
\[
  \varkappa^2 =
  \frac{\det(\dot{P}(u), \ddot{P}(u))^2}{\Vert \dot{P}(u) \Vert^3}.
\]
Evaluating its reciprocal at $u = u_V$ yields, after a lengthy calculation, the formula
\begin{equation}
  \label{eq:parameter2}
  \varrho^2 = \varkappa^{-2} =
  \frac{4(x_0y_1-x_0y_2-x_1y_0+x_1y_2+x_2y_0-x_2y_1)^4}{N^3}
\end{equation}
with
\[
  N = (x_0+x_2)^2 +(y_0+y_2)^2 - 4x_1(x_0-x_1+x_2) - 4y_1(y_0-y_1+y_2)
\]
for the \emph{squared parameter}~$\varrho^2$.

Finally, we compute the focus of the parabola \eqref{eq:10}. In our context, it is conveniently described as the intersection point of the two isotropic tangents of $p$ (c.f. \cite[p.~288]{UniverseConics}). These are the tangents with directions $(1,\pm\mathrm{i})$. They correspond to parabola points of respective parameter values 
\begin{equation*}
  u_{\pm} = \frac{y_0-y_1 \mp (x_0-x_1)\mathrm{i}}{y_0-2y_1+y_2 \mp (x_0-2x_1+x_2)\mathrm{i}}.
\end{equation*}
The intersection point of the two isotropic tangents is the focal point
\begin{equation}
  \label{eq:focus}
  F = \frac{1}{(x_0-2x_1+x_2)^2+(y_0-2y_1+y_2)^2}(f_1,f_2)
\end{equation}
where
\begin{multline*}
  f_1 =
  x_0^2x_2-x_0x_1^2-2x_0x_1x_2+x_0x_2^2+2x_1^3-x_1^2x_2-x_1y_0^2\\
  -2x_1y_0y_1+4x_1y_0y_2+2x_1y_1^2-2x_1y_1y_2-x_1y_2^2
  \\
  +x_0(y_1-y_2)^2
  +x_1(y_0-y_2)^2
  +x_2(y_0-y_1)^2
\end{multline*}
and
\begin{multline*}
  f_2 =
  -x_0^2y_1-2x_0x_1y_1+4x_0x_2y_1+2x_1^2y_1-2x_1x_2y_1-x_2^2y_1\\
  +y_0^2y_2-y_0y_1^2-2y_0y_1y_2+y_0y_2^2+2y_1^3-y_1^2y_2
  \\
  +y_0(x_1-x_2)^2
  +y_1(x_0-x_2)^2
  +y_2(x_0-x_1)^2.
\end{multline*}

\section*{Acknowledgment}

This research is funded by the Ministry of Education and Science of Republic of North Macedonia: Bilateral Austrian-Macedonian scientific Project No.~20-8436/20 with the title ``Estimates for Ellipsoids in Classical and Non-Euclidean Geometries'' and by OeAD-GmbH Austria's Agency for Education and Internationalization, Project No.~MK05/2024, Project Title ``Estimates for Ellipsoids in Classical and Non-Euclidean Geometries''.

\bibliographystyle{elsarticle-num}

\end{document}